\documentclass{amsart}
\usepackage{latexsym, amssymb, amsmath, amsthm}
\usepackage[T1]{fontenc}
\usepackage{lmodern}
\usepackage{enumerate}
\usepackage{mathabx}
\usepackage{csquotes}
\usepackage[mathscr]{euscript}

\usepackage{parskip}

\usepackage{tikz}
\usetikzlibrary{trees}

\usepackage{pifont}

\newtheorem{theorem}{Theorem}[section]
\newtheorem{lemma}[theorem]{Lemma}

\newtheorem{proposition}[theorem]{Proposition}

\theoremstyle{definition}
\newtheorem{definition}[theorem]{Definition}
\theoremstyle{remark}
\newtheorem{remark}[theorem]{Remark}

\newbox\gnBoxA
\newdimen\gnCornerHgt
\setbox\gnBoxA=\hbox{$\ulcorner$}
\global\gnCornerHgt=\ht\gnBoxA
\newdimen\gnArgHgt
\def\gnmb #1{%
\setbox\gnBoxA=\hbox{$#1$}%
\gnArgHgt=\ht\gnBoxA%
\ifnum     \gnArgHgt<\gnCornerHgt \gnArgHgt=0pt%
\else \advance \gnArgHgt by -\gnCornerHgt%
\fi \raise\gnArgHgt\hbox{$\ulcorner$} \box\gnBoxA %
\raise\gnArgHgt\hbox{$\urcorner$}}

\title{Reflection ranks via infinitary derivations}
\author{James Walsh}
\thanks{Thanks to Antonio Montalb\'{a}n for comments on a draft. Thanks to Anton Freund and Henry Towsner for helpful correspondence. Special thanks to Fedor Pakhomov; this paper is a descendant of our joint work. Thanks also to anonymous referees for their very helpful comments.}
\address{Department of Philosophy, New York University}
\email{jmw534@nyu.edu}

\begin{document}

\begin{abstract}
There is no infinite sequence of $\Pi^1_1$-sound extensions of $\mathsf{ACA}_0$ each of which proves $\Pi^1_1$-reflection of the next. This engenders a well-founded ``reflection ranking'' of $\Pi^1_1$-sound extensions of $\mathsf{ACA}_0$. For any $\Pi^1_1$-sound theory $T$ extending $\mathsf{ACA}^+_0$, the reflection rank of $T$ equals the proof-theoretic ordinal of $T$. This provides an alternative characterization of the notion of ``proof-theoretic ordinal,'' which is one of the central concepts of proof theory. We provide an alternative proof of this theorem using cut-elimination for infinitary derivations.
\end{abstract}

\maketitle

\section{Introduction}

It is a well-known empirical phenomenon that natural axiomatic theories are well-ordered by their consistency strength \cite{walsh2021hierarchy}. However, consistency strength does \emph{not} well-order theories \emph{in general}; indeed, as a purely formal matter, the consistency strength ordering is neither linear nor well-founded.\footnote{Feferman writes that the well-foundedness ``can't be otherwise
given that at any point in studying it there are only a finite number of theories that have
been considered'' \cite{feferman2014continuum}. I believe that Feferman undersells the phenomenon; the natural theories include, e.g., each theory of the form $\mathsf{I\Sigma_n}$, and there are infinitely many of them. Logicians have also studied infinite hierarchies of natural bounding, determinacy, and comprehension axioms, among others.} Pakhomov and the author investigated the well-foundedness half of this phenomenon by studying a related notion of proof-theoretic strength, namely, $\Pi^1_1$-reflection strength. In fact, that $\Pi^1_1$-sound extensions of $\mathsf{ACA}_0$ can be \emph{ranked} according to this notion of relative strength \cite[Theorem 3.2]{pakhomov2018reflection}.\footnote{For another version of this theorem with an alternate proof, see \cite[Theorem 3.7]{walsh2024incompleteness}.} To state this result, we recall that $\mathsf{RFN}_{\Pi^1_1}(T)$ is a sentence expressing the $\Pi^1_1$-soundness of $T$; we also note that $T \prec_{\Pi^1_1} U$ if and only if $U \vdash \mathsf{RFN}_{\Pi^1_1}(T)$.\footnote{Note that the theories related by $\prec_{\Pi^1_1}$ must be understood intensionally to make sense of the formula $\mathsf{RFN}_{\Pi^1_1}(T)$. Hence, the $U$ before the turnstile is really the theory defined by $U$.}

\begin{theorem}\label{reflection-ranks}
The restriction of $\prec_{\Pi^1_1}$ to the $\Pi^1_1$-sound extensions of $\mathsf{ACA}_0$ is well-founded.
\end{theorem}

Furthermore---in a large swathe of cases---a theory's rank in this well-founded reflection ordering coincides with its proof-theoretic ordinal, where the \emph{proof-theoretic ordinal} of a theory $T$ is the supremum of the $T$-provably well-founded primitive recursive linear orders. Proof-theoretic ordinals emerged from Gentzen's consistency proof of $\mathsf{PA}$ and have become central objects of study in proof theory. Indeed, one of the most prominent research programs in in proof theory---known as \emph{ordinal analysis}---consists in calculating theories' proof-theoretic ordinals and extracting information from these calculations.

Before stating the main result connecting reflection ranks and proof-theoretic ordinals, let me remind the reader that $\mathsf{ACA}_0^+$ is the theory $\mathsf{ACA}_0+ \forall X(X^{(\omega)}$ exists$)$.\footnote{An alternate axiomatization is $\mathsf{ACA}_0+$ ``Every set is contained in an $\omega$-model of $\mathsf{ACA}_0$''.} The following originally appeared as \cite[Theorem 5.16]{pakhomov2018reflection}:

\begin{theorem}\label{main-old}
For $\Pi^1_1$-sound extensions $T$ of $\mathsf{ACA}_0^+$, the $\prec_{\Pi^1_1}$ rank of $T$ equals the proof-theoretic ordinal of $T$.
\end{theorem}

Theorem \ref{reflection-ranks} and Theorem \ref{main-old} jointly yield the following claims: (i) Sufficiently sound theories can be ranked according to their proof-theoretic strength, given the right notion of proof-theoretic strength. (ii) The ranks in this ordering are proof-theoretically familiar values, namely, the proof-theoretic ordinals. Hence, Theorem \ref{main-old} connects two distinct topics in proof theory: iterated reflection and ordinal analysis. In \cite{pakhomov2018reflection}, the focus was on the iterated reflection side, and Theorem \ref{main-old} was derived using rather complicated proofs of Schmerl-style conservation theorems. In this paper, we present a different proof of Theorem \ref{main-old}. Though this new proof still requires some reasoning about iterated reflection, Schmerl-style conservation theorems are replaced with a traditional ordinal analysis technique, namely, cut-elimination for infinitary derivations. Accordingly, this proof should make Theorem \ref{main-old} accessible to a broader audience of logicians. In addition, the existence of these two proofs strengthens the connections between these two complementary areas of proof theory.

Here is our plan for the rest of the paper. In \textsection \ref{prelim-section} we cover some preliminaries, including our treatment of reflection principles and ordinal notations. In \textsection \ref{cut-lemmas} we discuss the treatment of infinitary derivations within $\mathsf{ACA}_0$. In \textsection \ref{easy-direction} we prove the main theorem.


\section{Preliminaries}\label{prelim-section}

In this section we cover some preliminary material and definitions. First, we give the definitions of $\Pi^1_1$-reflection and its iterations. Next, we review the definition of pseudo-$\Pi^1_1$ theories. We then cover our treatment of ordinal notations. Next, we collect some facts about well-ordering principles. Finally, we review Schmerl's technique of reflexive induction.

\subsection{Reflection Principles}

We are concerned with recursively enumerable theories. Officially, a theory $T$ is a $\Sigma_1$ formula $\mathsf{Ax}_T(x)$ that is understood as a formula defining the (G\"{o}del numbers of) axioms of $T$ in the standard model of arithmetic, i.e., the set of axioms of T is $\{\varphi \mid \mathbb{N}\models \mathsf{Ax}_T(\varphi)\}$. Thus, we are considering theories intensionally, via their axioms,
rather than as deductively closed sets of formulas.

In what follows, note that we have also fixed a standard choice for the provability predicate $\mathsf{Pr}(x,y)$, which expresses that $x$ is a theorem of $y$. To be slightly more explicit, $\mathsf{Pr}(x,y)$ states that formula with G\"odel number $x$ is provable from the axiom system defined by the $\Sigma_1$ formula with G\"odel number $y$. For a theory $T$ (i.e., a $\Sigma_1$ formula defining a set of axioms), we will abbreviate this as $\mathsf{Pr}_T(x)$. Our definitions of reflection principles should be interpreted with this fixed provability predicate in mind.

To introduce the $\Pi^1_1$ reflection schema, we must first review some notation. In what follows, $\overline{m}$ is the standard numeral of the number $m$ and $\ulcorner\varphi\urcorner$ is the G\"odel number of the formula $\varphi$. The expression $\ulcorner \varphi(\dot{x_1},\dots,\dot{x_n})\urcorner$ denotes an elementary definable term for the function taking $k_1,\dots,k_n$ to the G\"odel number $\ulcorner\varphi(\overline{k_1},\dots, \overline{k_n} )\urcorner$ of the result of substituting the numerals $\overline{k_1},\dots, \overline{k_n}$ for the variables $x_1,\dots,x_n$ in $\varphi$.

Officially speaking, $\mathsf{RFN}_{\Pi^1_1}(T)$ is the schema:
$$\forall x_1\dots\forall x_n \Big(\mathsf{Pr}_T\big(\ulcorner\varphi(\dot{x_1},\dots,\dot{x_n})\urcorner\big) \to \varphi (x_1,\dots,x_n) \Big) \text{ for } \varphi\in \Pi^1_1.$$
In fact, since the $\Pi^1_1$ truth-definition is available in $\mathsf{ACA}_0$, this schema follows from a single instance of itself in $\mathsf{ACA}_0$; see \cite{pakhomov2018reflection} for discussion. Hence we may regard this schema as a single sentence.

\begin{remark}
    I will mostly be informal rather than adhering strictly for conventions for  G\"odel numbering, the dot notation, and so on. For instance, going forward, I will write the $\Pi^1_1$-reflection schema as follows:
$$\forall x_1\dots\forall x_n \Big(\mathsf{Pr}_T\big(\varphi(x_1,\dots,x_n\big) \to \varphi (x_1,\dots,x_n) \Big) \text{ for } \varphi\in \Pi^1_1.$$
However, ``officially speaking,'' the formulas should have dots and corner quotes in them.
\end{remark}

Suppose that we have fixed an ordinal notation system $\prec$ (in \textsection \ref{ordinal-notations} I will explicitly define what ordinal notations are). We may then define the formula $\mathsf{RFN}^\alpha_{\Pi^1_1}(\mathsf{ACA}_0)$ of $\alpha$-iterated uniform $\Pi^1_1$-reflection over $\mathsf{ACA}_0$ via the fixed point lemma:
$$\mathsf{ACA}_0 \vdash \mathsf{RFN}_{\Pi^1_1}^\alpha(\mathsf{ACA}_0) \leftrightarrow \forall \beta \prec \alpha \; \mathsf{RFN}_{\Pi^1_1}\big(\mathsf{ACA}_0+ \mathsf{RFN}_{\Pi^1_1}^\beta(\mathsf{ACA}_0)\big).$$

\begin{remark}
Note that in our previous paper \cite{pakhomov2018reflection} we worked with a different definition of theories axiomatized by iterated reflection wherein limit stages yielded infinitely axiomatized theories. However, for the purposes of this paper, it suffices to limit our attention to iterated reflection principles that are finitely axiomatized. So our definitions yield that $\mathsf{RFN}_{\Pi^1_1}^\lambda(\mathsf{ACA}_0)$ is a sentence for any limit $\lambda$.
\end{remark}


\subsection{Pseudo-$\Pi^1_1$ Theories}

In this paper we will examine theories in two different languages:
\begin{enumerate}
    \item $\mathcal{L}_X$ is the language of first-order arithmetic extended with one additional free set variable $X$; we also call this the \emph{pseudo-$\Pi^1_1$ language}. Formulas from this language are known as \emph{pseudo-$\Pi^1_1$ formulas.} That is, we call \emph{all} the $\mathcal{L}_X$ formulas pseudo-$\Pi^1_1$ formulas (even if they do not contain any occurrences of the free variable $X$). By a pseudo-$\Pi^1_1$ theory we mean a theory axiomatized by pseudo-$\Pi^1_1$ formulas. 
    \item $\mathcal{L}_2$ is the language of second-order arithmetic.
\end{enumerate} We consider the pseudo-$\Pi^1_1$ language to be a sublanguage of the language of second-order arithmetic by identifying each pseudo-$\Pi^1_1$ sentence $\mathsf{F}$ with the second-order sentence $\forall X \;\mathsf{F}$.

The theory $\mathsf{PA}(X)$ is the pseudo-$\Pi^1_1$ pendant of $\mathsf{PA}$. That is, $\mathsf{PA}(X)$ contains (i) the axioms of $\mathsf{PA}$ and (ii) induction axioms for all formulas in the language, including those with the set variable $X$.

$\mathsf{ACA}_0$ is sometimes regarded as a second-order pendant of $\mathsf{PA}$. Indeed, $\mathsf{ACA}_0$ is arithmetically conservative over $\mathsf{PA}$; see, e.g., \cite{hirschfeldt2015slicing, simpson2009subsystems}. In fact, $\mathsf{ACA}_0$ is conservative over $\mathsf{PA}(X)$ for pseudo-$\Pi^1_1$ formulas, so it is often possible to work in $\mathsf{PA}(X)$ and transfer results to $\mathsf{ACA}_0$. Moreover, every extension of $\mathsf{ACA}_0$ by a $\Pi^1_1$ sentence is conservative over a corresponding extension of $\mathsf{PA}(X)$ by pseudo-$\Pi^1_1$ formulas. The following statement originally occurred as \cite[Lemma 4.10]{pakhomov2018reflection}:

\begin{lemma}\label{conservativity}
$(\mathsf{ACA}_0)$ Let $\varphi(X)$ and $\psi(X)$ be pseudo-$\Pi^1_1$ formulas. Suppose that $\mathsf{ACA}_0+\forall X\;\varphi(X) \vdash \forall X\;\psi(X)$. Then $\mathsf{PA}(X) + \{ \varphi[\theta] : \theta \textrm{ is pseudo } \Pi^1_1  \} \vdash \psi(X)$.
\end{lemma}


\subsection{Ordinal Notations}\label{ordinal-notations}

A standard approach to treating ordinal notations in arithmetic is to \emph{fix} a notation system up to a particular recursive ordinal. However, this approach is insufficiently general for present purposes, since we want to formalize results that involve quantification over all ordinal notations. Thus, we will use an alternative approach to ordinal notations from \cite{pakhomov2018reflection}.

Officially, an ordinal notation system is a tuple $\langle \varphi(x),\psi(x,y),p,n\rangle$ where $\varphi,\psi\in\Delta_0$, $\varphi(n)$ is true according to $\mathsf{True}_{\Sigma_1}$ and $p$ is an $\mathsf{EA}$ proof of the fact that $\psi(x,y)$ defines a strict linear order on the set $\{x\mid \varphi(x)\}$.

We may then define a partial order $\prec$ on the set of all notation systems. Any tuples $\alpha=\langle \varphi,\psi,p,n\rangle$ and $\alpha'=\langle \varphi',\psi',p',n\rangle$ are $\prec$-incomparable if either $\varphi\neq\varphi'$ or $\psi\neq \psi'$ or $p\neq p'$. If $\alpha,\beta$ have the form $\alpha=\langle \varphi,\psi,p,n\rangle$ and $\beta=\langle \varphi,\psi,p,m\rangle$ we put $\alpha\prec\beta$ if $\mathsf{True}_{\Sigma_1}\big(\psi(n,m)\big)$ but $\mathsf{True}_{\Sigma_1}\big(\neg \psi(m,n)\big)$.

For an ordinal notation $\alpha$ the value $|\alpha|$ is either an ordinal or $\infty$. If the lower cone $\big( \{\beta \mid \beta\prec\alpha\},\prec\big)$ is well-founded, then $|\alpha|$ is the ordinal isomorphic to the well-ordering $\big( \{\beta \mid \beta\prec\alpha\},\prec \big)$. Otherwise, $|\alpha|=\infty$. That is, $|\alpha|$ is the well-founded $\prec$-rank of $\alpha$.

We will also work with ordinal notation systems that are given by some combinatorially defined system of terms and order on them. One example of such a system is the Cantor ordinal notation system up to $\varepsilon_0$. For the notations that we consider it will always be possible to formalize, within $\mathsf{EA}$, both their definition and the proof that the ordering is linear.

We will use expressions like $\omega^\alpha$ and $\varepsilon_\alpha$ where $\alpha$ is some ordinal notation system. Let us consider a notation system $\alpha=\langle \varphi(x),\psi(x,y),p,n\rangle$ and define the notation system $\omega^\alpha =\langle \varphi ' (x),\psi'(x,y),p',n'\rangle$. We want the order $\prec_{\omega^\alpha}$ to be the order on the terms $\omega^{a_1} + \dots + \omega^{a_k}$, where $a_1\succeq_{\alpha}\dots \succeq_{\alpha} a_k$. And the order $\prec_{\omega^\alpha}$ is defined as the usual order on Cantor normal forms, where we compare $a_i$ by the order $\prec_\alpha$. By arithmetizing this definition of $\prec_{\omega^\alpha}$ we get $\varphi',\psi'$, and $p'$. We put $n'$ to be the number encoding the term $\omega^n$. Note that, according to this definition, $\alpha$ and $\omega^\alpha$ are $\prec$-incomparable. However, if $\alpha\prec\beta$, then $\omega^\alpha\prec \omega^\beta$.

The definition of the notation system $\varepsilon_\alpha$ is similar to that of $\omega^\alpha$. The system of terms for $\varepsilon_\alpha$ consists of nested Cantor normal forms built up from 0 and elements $\varepsilon_a$, $a\in\mathsf{dom}(\prec_\alpha)$. The comparison of nested Cantor normal forms is defined in the standard fashion, where we compare elements $\varepsilon_a$ and $\varepsilon_b$ as $a\prec_\alpha b$. For more information on our treatment of ordinal notation systems, we refer the reader to \cite[\textsection 2.2]{pakhomov2018reflection}.

\subsection{Well-ordering Principles}

Before continuing we record two facts:
\begin{theorem}
\label{preserves} Provably in $\mathsf{ACA}_0$, for any linear order $\mathcal{X}$, if $\mathcal{X}$ is well-ordered then so is $\omega^{\mathcal{X}}$.
\end{theorem}
\begin{theorem}
\label{mm}
 Provably in $\mathsf{ACA}_0^+$, for any linear order $\mathcal{X}$, if $\mathcal{X}$ is well-ordered then so is $\varepsilon_{\mathcal{X}}$.
\end{theorem}
For a proof of the first see \cite{girard1987proof}; for a proof of the second see \cite{marcone2011veblen}. These theorems are stated quite generally, in terms of arbitrary linear orders. They straightforwardly apply to ordinal notation systems, given our treatment of ordinal notation systems.

    Since Theorem \ref{preserves} records a $\Sigma_1$ fact (namely, that $\mathsf{ACA}_0$ proves something), it too is provable in $\mathsf{ACA}_0$. Hence we have:
    $$\mathsf{ACA}_0\vdash\text{``$\mathsf{ACA}_0 \vdash \forall \mathcal{X}\big( \mathsf{WO}(\mathcal{X}) \to \mathsf{WO}(\omega^\mathcal{X})\big)$.''}$$
    Thus, $\mathsf{ACA}_0$ proves that its own provably well-founded ordinals must be closed under exponentiation base $\omega$. Hence, we have the following:
    
    \begin{proposition}[$\mathsf{ACA}_0$]\label{clarify}
If any extension $T$ of $\mathsf{ACA}_0$ proves $\mathsf{WO}(\varepsilon_\delta)$ then, for every $\alpha<\varepsilon_{\delta+1}$, $T$ proves $\mathsf{WO}(\alpha)$.
    \end{proposition}


\subsection{Reflexive induction}\label{reflexive-section}

We will use Schmerl's technique of \emph{reflexive induction} \cite{schmerl1979fine}. Reflexive induction is a way of simulating large amounts of transfinite induction in weak theories. The technique is facilitated by the following theorem; we include the proof since it is so short.

\begin{theorem}[Schmerl]
Let $T$ be a recursively axiomatized theory containing $\mathsf{EA}$. Suppose $T \vdash \forall \alpha \Big( \mathsf{Pr}_T\big(\forall \beta\prec \alpha \varphi(\beta)\big) \to \varphi(\alpha) \Big).$
Then $T\vdash \forall \alpha \varphi(\alpha)$.
\end{theorem}

\begin{proof}
Suppose that $T \vdash \forall \alpha \Big( \mathsf{Pr}_T\big(\forall \beta\prec \alpha \varphi(\beta)\big) \to \varphi(\alpha) \Big).$ We infer that:
\begin{flalign*}
T &\vdash \forall \alpha \mathsf{Pr}_T\big(\forall \beta\prec \alpha \varphi(\beta)\big) \to \forall \alpha \varphi(\alpha)\\
T &\vdash \mathsf{Pr}_T\big(\forall \alpha \varphi(\alpha)\big) \to \forall \alpha \varphi(\alpha)
\end{flalign*}
L\"{o}b's Theorem then yields $T\vdash \forall \alpha \varphi(\alpha)$.
\end{proof}


\section{Infinitary Derivations}\label{cut-lemmas}

Before proving the main theorem, we must first prove, within a theory of the form $\mathsf{ACA}_0+\mathsf{WO}(\varepsilon_\alpha)$, the $\Pi^1_1$-soundness of $\mathsf{ACA}_0+\mathsf{WO}(\varepsilon_\gamma)$ uniformly for all $\gamma<\alpha$. That is, we must prove the following:

\begin{theorem}[$\mathsf{ACA}_0$]\label{to be applied}
$\mathsf{ACA}_0+\mathsf{WO}(\varepsilon_\alpha)$ proves that for every $\gamma<\alpha$, $\mathsf{ACA}_0+\mathsf{WO}(\varepsilon_\gamma)$ is $\Pi^1_1$-sound.
\end{theorem}

We do this by carrying out, within $\mathsf{ACA}_0+\mathsf{WO}(\varepsilon_\alpha)$, ordinal analyses of the theories $\mathsf{ACA}_0+\mathsf{WO}(\varepsilon_\gamma)$ for all $\gamma<\alpha$. The strategy of the proof is this. We consider a proof of a $\Pi^1_1$ formula $\varphi$ from the axiom system $\mathsf{ACA}_0+\mathsf{WO}(\varepsilon_\gamma)$. We then embed this proof into a $\Pi^1_1$-complete infinitary proof calculus for $\omega$-logic. From the cut-elimination theorem for $\omega$-logic, we infer that $\varphi$ has a cut-free proof in $\omega$-logic whose height is less than $\varepsilon_\alpha$. So we may restrict our attention to the fragment of $\omega$-logic consisting only of proofs whose height is less than $\varepsilon_\alpha$. We can prove the soundness of this fragment using transfinite induction along $\varepsilon_\alpha$. The truth of $\varphi$ follows.

Accordingly, we must explain how it is that we can reason about infinitary proof systems \emph{within} $\mathsf{ACA}_0$. That is the goal of this section. First, we present the standard definition of infinitary proof systems; though this definition eludes $\mathsf{ACA}_0$, it should help the reader develop an intuitive feel for infinitary derivations. Then we present a definition of infinitary derivations within $\mathsf{ACA}_0$. Afterwards, we prove a soundness lemma in $\mathsf{ACA}_0$, namely, that well-founded cut-free proofs yield only true $\Pi^1_1$ conclusions. We then explain how we can reason about transformations of infinitary proofs within $\mathsf{ACA}_0$. We conclude this section by collecting a few crucial lemmas about infinitary derivations that are necessary for proving the main theorem.

\subsection{A Semi-formal Calculus}

The standard definition of infinitary proof systems uses transfinite recursion, which is beyond the scope of $\mathsf{ACA}_0$. Nevertheless, I will start by presenting this standard definition, since I expect that it will better aid readers' intuitions Only afterwards do we discuss how to arithmetize this definition.

\begin{definition}
\label{proofdefinition}
Let $diag(\mathbb{N})$ be the atomic diagram of $\mathbb{N}$ in the signature $(0,1,+,\times)$. We define the relation $\vdash^\alpha_\rho \Delta$ via transfinite recursion.

$\vdash^\alpha_\rho \Delta$ holds if at least one of the following conditions is met:
\begin{itemize}
\item $\Delta\cap diag(\mathbb{N})\neq\emptyset$.
\item For some $t$ and $s$ such that $t^\mathbb{N}=s^\mathbb{N}$, $\Delta$ contains both formulas $s\notin X, t\in X$.
\item $\Delta=\Gamma,A_1\wedge A_2$ and for all $i\in \{1,2\}$ there is $\alpha_i\prec \alpha$ such that $\vdash^{\alpha_i}_\rho \Gamma, A_i $.
\item $\Delta=\Gamma,\forall x \; A(x)$ and for all $i\in \mathbb{N}$ there is $\alpha_i\prec \alpha$ such that $\vdash^{\alpha_i}_\rho \Gamma, A(i) $.
\item $\Delta=\Gamma,A_1\lor A_2$ and for some $i\in \{1,2\}$ there is $\alpha_i\prec \alpha$ such that $\vdash^{\alpha_i}_\rho \Gamma, A_i$.
\item $\Delta=\Gamma,\exists x \; A(x)$ and for some $i\in \mathbb{N}$ there is $\alpha_i\prec \alpha$ such that $\vdash^{\alpha_i}_\rho \Gamma, A(i)$.
\item For some sentence $F$ such that $\mathit{rk}(F)<\rho$ and for some $\alpha_1,\alpha_2\prec \alpha$, $\vdash^{\alpha_1}_\rho \Delta, F$ and $\vdash^{\alpha_2}_\rho \Delta, \neg F$.
\end{itemize}
\end{definition}

\subsection{Arithmetizing Infinitary Derivations}

In this section we give a definition of infinitary derivations that works in $\mathsf{ACA}_0$. I follow the approach in \cite[\textsection 4.2]{walsh2024incompleteness}.

The basic idea behind our definition of infinitary proofs in $\mathsf{ACA}_0$ is that infinitary proofs are $\omega$-branching trees. Each node in the proof is \emph{tagged} with a sequent, a rule, an ordinal notation, and a cut rank:
\begin{definition}[$\mathsf{ACA}_0$]
Let $\mathsf{SEQ}$ be the set of finite sequents, i.e., sets of formulas in (Tait calculus) normal form in the signature $(0,1,+,\times)$. Let $$\mathsf{RULE} =\{\mathsf{AxM}, \mathsf{AxL}, \wedge, \vee, \forall, \exists, \mathsf{CUT}, \mathsf{REP} \}.$$
\end{definition}

We demand that the trees satisfy \emph{local correctness conditions}. The local correctness conditions merely say that if a node is tagged with a sequent $\Delta$ and rule $R$, then the premises of that node are tagged with sequents that are correct for the rule $R$. Buchholz essentially introduces these local correctness conditions (changed only slightly here) in \cite[Definitions 2.1--2.3]{buchholz1991notation}.

\begin{definition}[$\mathsf{ACA}_0$]\label{local}
Let $(\Delta,R)\in \mathsf{SEQ}\times \mathsf{RULE}$ and let $(\Delta)_{i\in I}$ be a sequence of sequents (the premises of $\Delta$). We say that $(\Delta,R)$ and $(\Delta)_{i\in I}$ jointly satisfy the \emph{local correctness conditions} if each of the following holds:
\begin{itemize}
    \item[(AxM)] If $R=\mathsf{AxM}$ then $\Delta\cap\mathsf{Diag}(\mathbb{N})\neq\emptyset.$
    \item[(AxL)] If $R=\mathsf{AxL}$ then there are $t^\mathbb{N}=s^\mathbb{N}$ such that $s\notin X,t\in X\in \Delta$.
    \item[($\wedge$)] If $R=\wedge$ then $I=\{1,2\}$ and for some $A_1$ and $A_2$: 
    $$ \text{$A_1\wedge A_2\in\Delta$ and for all $i\in\{1,2\}$, $\Delta_i \subseteq \Delta,A_i$.}$$
    \item[($\wedge$)] If $R=\vee$ then $I\subseteq\{1,2\}$ and for some $A_1$ and $A_2$: 
    $$ \text{$A_1\vee A_2\in\Delta$ and for some $i\in\{1,2\}$, $\Delta_i \subseteq \Delta,A_i$.}$$
    \item[($\forall$)] If $R=\forall$ then $I=\mathbb{N}$ and for some $\forall xA(x)$: 
    $$ \text{$\forall xA(x)\in\Delta$ and for all $i\in \mathbb{N}$, $\Delta_i \subseteq \Delta,A(i)$.}$$
    \item[($\exists$)] If $R=\exists$ then $I\subseteq\mathbb{N}$ and for some $\exists xA(x)$: 
    $$ \text{$\exists xA(x)\in\Delta$ and for some $i\in \mathbb{N}$, $\Delta_i \subseteq \Delta,A(i)$.}$$
    \item[(CUT)] If $R=\mathsf{CUT}$ then $I=\{1,2 \}$ and for some $A$: 
    $$ \text{$\Delta_1 \subseteq \Delta,A$ and $\Delta_2\subseteq \Delta,\neg A$.}$$
    \item[(REP)] If $R=\mathsf{REP}$ then $I=\{1\}$ and $ \text{$\Delta_1 = \Delta$.}$
\end{itemize}
\end{definition}

The repetition rule \textsf{REP} must be included for technical reasons. With the definition of the local correctness conditions on board, we may specify which $\omega$-branching trees are infinitary proofs.

\begin{definition}[$\mathsf{ACA}_0$]\label{inf-proof}
Given an ordinal notation system $\prec$, an \emph{infinitary proof} is an $\omega$-branching tree where each node is labeled by a quadruple $(\Delta,R,\alpha,n)$ consisting of a sequent $\Delta$, rule $R$, an ordinal notation $\alpha$, and a number $n$ such that:
\begin{enumerate}
    \item The ordinal labels strictly descend (in the sense of $\prec$) from the root towards the axioms.
    \item If a conclusion $(\Delta,\mathsf{CUT},\alpha,n)$ has premises with sequent labels $\Delta_1,A$ and $\Delta_2,\neg A$ then $n>\mathsf{rk}(A)$; number labels cannot ascend from the root towards the axioms.
\item The local correctness conditions from Definition \ref{local} are satisfied.
\end{enumerate}
\end{definition}

Now let's introduce some notation.
\begin{definition}[$\mathsf{ACA}_0$]
    If $P$ is an infinitary proof such that, for some $R$, $(\Delta, R, \alpha, n) \in P$, we write $P\vdash^\alpha_n \Delta$. If there is an infinitary proof $P$ such that $P\vdash^\alpha_n \Delta$, we write $\vdash^\alpha_n \Delta$.
\end{definition}

\subsection{The Soundness Lemma}

Before proving the main lemma, we collect a number of facts concerning infinitary derivations. The first such fact is that our semi-formal system is $\Pi^1_1$-sound. Given the \emph{standard} Definition \ref{proofdefinition} of infinitary proof calculi, one can establish this easily by transfinite induction. Of course, we are not officially using Definition \ref{proofdefinition}. Hence, we take more care than is typical, and provide the proof in $\mathsf{ACA}_0$ using the formalization of proofs given in Definition \ref{inf-proof}.

\begin{lemma}
\label{correctness} 
For each $\Pi^0_{<\omega}$ formula $\varphi(X,\vec{x})$, the theory $\mathsf{ACA}_0$ proves:   If $\alpha$ is well-founded and $\vdash^\alpha_0 \varphi(X,\vec{p})$, for some vector $\vec{p}$ of number parameters, then $\forall X\;\varphi(X,\vec{p})$.
\end{lemma}

\begin{proof}
We reason in $\mathsf{ACA}_0$. Assume that $\forall X \varphi(X,\vec{p})$ is false but that $\vdash^\alpha_0 \varphi(X,\vec{p})$. We will show that $\alpha$ is ill-founded.

Since $\forall X \varphi(X,\vec{p})$ is false, there is some set $A$  such that $\varphi(A,\vec{p})$. Consider some $P$ such that $P\vdash^\alpha_0\varphi(X,\vec{p})$. That is, $P$ is an $\omega$-branching tree that is labeled with quadruples according to Definition \ref{inf-proof} whose root has label $(\varphi(X,\vec{p}),R,\alpha,0)$.

We consider the set $\mathcal{U}$ of all pairs $\langle\beta,\Gamma(X)\rangle$ such that each of the following holds:
\begin{enumerate}[i]
    \item $\beta\preceq \alpha$.
    \item $\Gamma(X)$ is a sequent consisting of numerical variants of subformulas of $\varphi(X)$.
    \item $\Gamma(A)$ is false.
    \item Some subproof of $P$ witnesses $\vdash^\beta_0\Gamma(X)$, i.e., for some $R'$:
    $$(\Gamma(X),R', \beta,0)\in P.$$
\end{enumerate}
Note that this construction is valid in $\mathsf{ACA}_0$. Indeed, $\varphi(X,\vec{x})$ is fixed, so the truth of the sequents $\Gamma(A)$ that could appear in members of $\mathcal{U}$ can be expressed in $\mathsf{ACA}_0$. Notice that $\langle\alpha,\varphi(X,\vec{p})\rangle$ is in $\mathcal{U}$ and that whenever $\langle \beta,\Gamma(X)\rangle $ is in $\mathcal{U}$ there is some $\langle \beta',\Gamma'(X)\rangle$ in $\mathcal{U}$ with $\beta'\prec\beta$ (we could always choose $\langle\beta',\Gamma'(X)\rangle$ to be one of the premises used to derive $\langle\beta,\Gamma(X)\rangle$).  Thus the set of all $\beta$ that appear in elements in $\mathcal{U}$ is a set of denotations of ordinals without a least element.
\end{proof}


\subsection{Transformation Lemmas}

In this subsection we collect a few more lemmas. These lemmas are entirely standard. However, they require that we reason about \emph{transformations} of infinitary proofs within $\mathsf{ACA}_0$. Proof theorists have developed many methods for arithmetizing such transformations for infinitary proofs; examples include \cite{girard1987proof, kreisel1965mathematical, schwichtenberg1977proof}. One particularly elegant way of formalizing transformations of infinitary proofs was developed by Buchholz. In \cite[\textsection 5]{buchholz1991notation}, Buchholz describes a uniform way of developing notation systems for infinitary derivations. The definition is given entirely by primitive recursion, so it can be implemented within $\mathsf{ACA}_0$ (and much weaker systems). The term representing a proof in this notation system encodes the information in its root, i.e., its sequent, the rule it was inferred with, its ordinal bound, and its cut rank. For any such term, the local correctness of the proof it encodes is provable in weak systems (e.g., $\mathsf{PRA}$). In \cite[\textsection 6]{buchholz1991notation}, Buchholz shows that the proof transformations required for cut-elimination (e.g., the Reduction Lemma and Elimination Lemma) can be proven by reasoning about these term systems. 

Officially speaking, we adopt the Buchholz the approach to reasoning about proof transformations. To reiterate, we have made two choices:
\begin{enumerate}
    \item Infinitary proofs are defined in $\mathsf{ACA}_0$ as $\omega$-branching trees meeting the conditions laid out in Definition \ref{inf-proof}.
    \item Transformations of such infinitary proofs is carried out in the style of \cite{buchholz1991notation}; given an infinitary proof, we may concoct a term system for which the relevant transformations are provable.
\end{enumerate}

For the rest of the section, we collect some standard results that are typically proved via proof transformations. Instead of giving proofs, we give citations to the relevant literature.

\subsubsection{The Induction Lemma}

Since $\omega$-logic extends the truth-definition for arithmetical sentences (i.e., those without set variables), each arithmetical sentence has a cut-free proof in $\omega$-logic with finite height. A bit more work yields the \emph{Induction Lemma}.
\begin{lemma}[{\cite[Equation 34 following Lemma 2.1.2.4]{pohlers1998subsystems}}]
\label{bound}
For each axiom $A$ of $\mathsf{PA}$, there is an $\alpha\leq \omega+4$ such that $\vdash^\alpha_0 A$.
\end{lemma}

\subsubsection{The Converse Boundedness Lemma}

For an arithmetically definable relation $\prec$ and pseudo-$\Pi^1_1$ formula $\varphi$, we provide the following definitions:
$$\mathsf{field}(\prec):= \{ x \mid \exists y (x\prec y \vee x \prec y) \} $$
$$\mathsf{Prog}(\prec,\varphi):=\forall x\Big( \big(x \in\mathsf{field}(\prec) \wedge \forall y (y\prec x\to \varphi(y)) \big) \to \varphi(x) \Big)$$
$$\mathsf{TI}(\prec,\varphi):=\mathsf{Prog}(\prec,\varphi) \to \forall x \in \mathsf{field}(\prec)\; \varphi(x).$$ 

The statement that we need is a modest strengthening of \cite[Theorem 1.3.10]{pohlers1998subsystems}, which we call the \emph{Converse Boundedness Lemma}.
\begin{lemma}\label{ordinals} 
For any pseudo-$\Pi^1_1$ formula $\varphi$, if $\mathsf{otyp}(\prec)$ is a limit ordinal, there is an $\alpha\leq \mathsf{otyp}(\prec)+2$ such that $\vdash^\alpha_0\mathsf{TI}(\prec,\varphi)$.
\end{lemma}
To convert the proof of \cite[Theorem 1.3.10]{pohlers1998subsystems} into a proof of Lemma \ref{ordinals}, simply replace all occurrences of $t\in X$ with $\varphi(t)$.

We assume that $\prec$ is primitive recursive and its field is all of $\mathbb{N}$.

\subsubsection{Two Definitions}

Before collecting more lemmas, we recall two standard definitions. We present the definition for the reader's convenience.

\begin{definition}
We define the \emph{symmetric sum} of ordinals (in Cantor normal form) $\alpha:= \omega^{\alpha_1}+\dots +\omega^{\alpha_n}$ and $\beta= \omega^{\alpha_{n+1}}+\dots +\omega^{\alpha_m}$ as follows:
$$\alpha \# \beta := \alpha_{\pi(1)}+\dots +\alpha_{\pi(m)}$$
where $\pi$ is a permutation of the numbers $\{1,\dots,m\}$ such that:
$$1\leq i \leq j \leq m \Rightarrow \alpha_{\pi(i)}\geq \alpha_{\pi(j)}.$$
\end{definition}

\begin{remark}
Note that $\alpha\# \beta=\beta\#\alpha$.
\end{remark}

\begin{remark}\label{symmetric}
    The notion of symmetric sum can be arithmetized for ordinal notations \emph{as long as} we select ordinal notations from the same notation system. Hence, whenever we deploy a result that relies on this notion (such as Lemma \ref{embedding} and Lemma \ref{elimination}), we assume that all notations are drawn from the same notation system.
\end{remark}

\begin{definition}
The rank $\mathsf{rk}(\varphi)$ of a formula $\varphi$ is the number of logical symbols that occur in $\varphi$.
\end{definition}

\subsubsection{The Reduction Lemma}

With these definitions on board we present the \emph{Reduction Lemma}.
\begin{lemma}[{\cite[Lemma 2.1.2.7]{pohlers1998subsystems}}]
\label{embedding} 
Assume $\vdash^\alpha_\rho \Delta,F$ and $\vdash^\beta_\rho \Gamma,\neg F$, as well as $\mathsf{rk}(F)=\rho$. Then $\vdash^{\alpha\#\beta}_\rho \Delta,\Gamma$.
\end{lemma}
As stated in Remark \ref{symmetric}, we assume that $\alpha$ and $\beta$ are drawn from the same ordinal notation system so that the natural sum of the ordinal notations $\alpha$ and $\beta$ is well-defined.

\subsubsection{The Elimination Lemma}

Finally, we must state the \emph{Elimination Lemma}. Pohlers states the Elimination Lemma for genuine infinitary proof trees as follows:

\begin{theorem}[{\cite[Theorem 2.1.2.8]{pohlers1998subsystems}}]
 If $\vdash^{\alpha}_{\rho+1} \Delta$, then $\vdash^{2^\alpha}_{\rho} \Delta$.
\end{theorem}

Since our definition of infinitary proofs includes the \textsf{REP} rule, the bound $2^\alpha$ does not quite work for our formalization. Hence, when we formalize the result in $\mathsf{ACA}_0$, we must restate it slightly.
\begin{lemma}[$\mathsf{ACA}_0$]
\label{elimination}  If $\vdash^{\alpha}_{\rho+1} \Delta$, then $\vdash^{3^\alpha}_{\rho} \Delta$.
\end{lemma}

Let me summarize what this shows. Given any ordinal notation system (closed under exponentiation base 3), there is a uniform definition in $\mathsf{ACA}_0$ of the proofs labeled by notations from this system. Moreover, there is a uniform definition of the term systems arising from such proofs. So we may reason quite generally about transformations of the proofs that arise from an arbitrary ordinal notation system. Hence, in $\mathsf{ACA}_0$, we may prove, for instance, that for any proof $P\vdash^\alpha_{\rho+1}\Delta$, there is a proof $P'\vdash^{3^\alpha}_\rho \Delta$.

\section{The Proof}\label{easy-direction}

Theorem \ref{reflection-ranks} states that the restriction of $\prec_{\Pi^1_1}:=\{(T,U) \mid U \vdash \mathsf{RFN}_{\Pi^1_1}(T)\}$  to the $\Pi^1_1$-sound extensions of $\mathsf{ACA}_0$ is well-founded. We write $|T|_{\mathsf{ACA}_0}$ to denote the rank of a theory in this ordering. We call $|T|_{\mathsf{ACA}_0}$ the \emph{reflection rank} of $T$. Our goal is to connect reflection ranks with \emph{proof-theoretic ordinals}. We write $|T|_{\mathsf{WO}}$ to denote the proof-theoretic ordinal of a theory $T$, i.e., the supremum of the $T$-provably well-founded primitive recursive linear orders.

We ultimately derive the main theorem from the following equivalence:
$$\mathsf{ACA}_0+\mathsf{RFN}^{1+\alpha}_{\Pi^1_1}(\mathsf{ACA}_0) \equiv \mathsf{ACA}_0+ \mathsf{WO}(\varepsilon_{\alpha}).$$
We prove each inclusion of this equivalence separately. In \textsection \ref{easy} we establish the easy direction, namely:
\begin{equation}
\label{eq:bigcirc} 
\tag{$\bigcirc$} 
\mathsf{ACA}_0+\mathsf{RFN}^{1+\alpha}_{\Pi^1_1}(\mathsf{ACA}_0) \vdash \mathsf{WO}(\varepsilon_{\alpha}).
\end{equation}
Finally, in \textsection \ref{hard-direction} we establish the hard direction: 
\begin{equation}
\label{eq:bigtriangleup} 
\tag{$\bigtriangleup$} 
\mathsf{ACA}_0+ \mathsf{WO}(\varepsilon_{\alpha}) \vdash \mathsf{RFN}^{1+\alpha}_{\Pi^1_1}(\mathsf{ACA}_0).
\end{equation}

\subsection{The Easy Direction}\label{easy}
In this section we prove $(\bigcirc)$. We prove the claim in $\mathsf{ACA}_0$ by Schmerl's \cite{schmerl1979fine} technique of \emph{reflexive induction}; see \textsection \ref{reflexive-section}.

\begin{theorem}\label{suggestion}
$\mathsf{ACA}_0\vdash \forall\alpha (\mathsf{ACA}_0+\mathsf{RFN}^{1+\alpha}_{\Pi^1_1}(\mathsf{ACA}_0) \textrm{ proves } \mathsf{WO}(\varepsilon_{\alpha})).$
\end{theorem}

\begin{proof}
We let $A(\alpha)$ denote the claim that $\mathsf{ACA}_0+\mathsf{RFN}^{1+\alpha}_{\Pi^1_1}(\mathsf{ACA}_0) \vdash \mathsf{WO}(\varepsilon_{\alpha})$. We want to prove that $\mathsf{ACA}_0\vdash \forall\alpha A(\alpha)$. By reflexive induction it suffices to prove that $\mathsf{ACA}_0 \vdash \forall\alpha (\mathsf{Pr}_{\mathsf{ACA}_0}(\forall\gamma<\alpha \; A(\gamma)) \rightarrow A(\alpha))$.

\textbf{Reason in} $\mathsf{ACA}_0$. We suppose the \emph{reflexive induction hypothesis}, which is that $\mathsf{Pr}_{\mathsf{ACA}_0}(\forall\gamma<\alpha \; A(\gamma))$, i.e.,
$$ \mathsf{Pr}_{\mathsf{ACA}_0}\Big(\forall\gamma<\alpha \big(\mathsf{ACA}_0+\mathsf{RFN}^{1+\gamma}_{\Pi^1_1}(\mathsf{ACA}_0) \textrm{ proves } \mathsf{WO}(\varepsilon_{\gamma})\big)\Big) . $$
We need to show that $A(\alpha)$, i.e., $\mathsf{ACA}_0+\mathsf{RFN}^{1+\alpha}_{\Pi^1_1}(\mathsf{ACA}_0) \vdash \mathsf{WO}(\varepsilon_{\alpha})$.

To prove the claim we reason as follows:
\begin{flalign*}
\mathsf{ACA}_0 &\vdash \forall\gamma<\alpha (\mathsf{ACA}_0+\mathsf{RFN}^{1+\gamma}_{\Pi^1_1}(\mathsf{ACA}_0) \textrm{ proves } \mathsf{WO}(\varepsilon_{\gamma}))\\
& \textrm{ by the reflexive induction hypothesis.}\\
\mathsf{ACA}_0 + \mathsf{RFN}^{1+\alpha}_{\Pi^1_1}(\mathsf{ACA}_0) &\vdash \forall\gamma<\alpha (\mathsf{ACA}_0+\mathsf{RFN}^{1+\gamma}_{\Pi^1_1}(\mathsf{ACA}_0) \textrm{ proves } \mathsf{WO}(\varepsilon_{\gamma}))\\
\mathsf{ACA}_0 + \mathsf{RFN}^{1+\alpha}_{\Pi^1_1}(\mathsf{ACA}_0) &\vdash \forall\gamma<\alpha \; \forall\delta<\varepsilon_{\gamma+1} (\mathsf{ACA}_0+\mathsf{RFN}^{1+\gamma}_{\Pi^1_1}(\mathsf{ACA}_0) \textrm{ proves } \mathsf{WO}(\delta))\\
&\textrm{ by Proposition \ref{clarify}}\\
\mathsf{ACA}_0 + \mathsf{RFN}^{1+\alpha}_{\Pi^1_1}(\mathsf{ACA}_0) &\vdash \forall\gamma<\alpha \; \forall\delta<\varepsilon_{\gamma+1} \; \mathsf{WO}(\delta)\\
&\textrm{ since well-foundedness is $\Pi^1_1$.}\\
\mathsf{ACA}_0 + \mathsf{RFN}^{1+\alpha}_{\Pi^1_1}(\mathsf{ACA}_0) &\vdash  \forall\delta<\varepsilon_{\alpha} \; \mathsf{WO}(\delta)\\
\mathsf{ACA}_0 + \mathsf{RFN}^{1+\alpha}_{\Pi^1_1}(\mathsf{ACA}_0) &\vdash  \mathsf{WO}(\varepsilon_\alpha)
\end{flalign*}
This completes the proof of the theorem.
\end{proof}


\subsection{The Main Lemma}

Before presenting the proof of the main lemma for \eqref{eq:bigtriangleup}, we introduce the following notation for convenience: We write $\mathsf{PA}^\alpha$ to denote the pseudo-$\Pi^1_1$ pendant of $\mathsf{ACA}_0 +\mathsf{WO}(\alpha)$, i.e.,
$$\mathsf{PA}^\alpha:=\mathsf{PA}(X)+\{ \mathsf{TI}(\alpha,\varphi)\mid \varphi \textrm{ is pseudo-$\Pi^1_1$} \}.$$

\begin{lemma}[$\mathsf{ACA}_0$] 
\label{analysis}For any ordinal notation $\alpha$:
$$\mathsf{ACA}_0 + \mathsf{WO}(\varepsilon_\alpha) \vdash \forall\gamma<\alpha \; \mathsf{RFN}_{\Pi^1_1}\big(\mathsf{ACA}_0+\mathsf{WO}(\varepsilon_\gamma)\big).$$
\end{lemma}

\begin{proof}
Recall that $\Pi^1_1$-reflection follows from a single instance of itself in $\mathsf{ACA}_0$. Hence, there is a fixed $\Pi^1_1$ formula $\varphi(\vec{x})$ such that it suffices to prove in $\mathsf{ACA}_0+\mathsf{WO}(\varepsilon_\alpha)$ that for $\gamma<\alpha$ and number parameters $\vec{p}$ we have
\begin{equation}
\label{RFN_from_WO_main_claim}
\mathsf{Pr}_{\mathsf{ACA}_0+\mathsf{WO}(\varepsilon_\gamma)}(\varphi(\vec{p}))\to \varphi(\vec{p}).\end{equation}

Suppose that $\mathsf{ACA}_0+\mathsf{WO}(\varepsilon_\gamma) \vdash \varphi(\vec{p})$ for some $\gamma<\alpha$. By Lemma \ref{conservativity}, we infer that $\mathsf{PA}^{\varepsilon_\gamma}\vdash \varphi^\star(\vec{p})$ for the pseudo-$\Pi^1_1$ version $\varphi^\star(\vec{p})$ of $\varphi(\vec{p})$. This is to say that 
$$ \mathsf{PA}\vdash \neg \mathsf{TI}(\varepsilon_\gamma, \psi_1)\vee...\vee\neg\mathsf{TI}(\varepsilon_\gamma,\psi_k)\vee\varphi^\star(\vec{p})$$ 
for some pseudo-$\Pi^1_1$ formulas $\psi_1,...,\psi_k$. Then there is a proof in pure first-order logic of the sequent 
$$ \neg A_1,...,\neg A_n, \bigvee_{i\leq k} \neg\mathsf{TI}(\varepsilon_\gamma,\psi_i),\varphi^\star(\vec{p})$$ 
for some axioms $A_1,\dots,A_n$ of $\mathsf{PA}$. Since this is a proof in pure first-order logic, this means that there are finite $h$ and $\rho$, such that:
$$\vdash^h_\rho \neg A_1,\dots,\neg A_n, \bigvee_{i\leq k} \neg\mathsf{TI}(\varepsilon_\gamma,\psi_i),\varphi^\star(\vec{p}).$$
By Lemma \ref{bound}, for each $A_i$, there is a cut-free proof $A_i$ whose height is less than $\omega+5$. Hence, $n$ applications of the cut rule yields:
$$\vdash^{\omega+4+n}_{\sigma} \bigvee_{i\leq k} \neg\mathsf{TI}(\varepsilon_\gamma,\psi_i),\varphi^\star(\vec{p})$$
where $\sigma:=\mathsf{max}\{\rho, \mathsf{rk}(A_1),\dots,\mathsf{rk}(A_n) \}+1$

By Lemma \ref{ordinals}, there is a cut-free proof of $\mathsf{TI}(\varepsilon_\gamma, \bigvee_{i\leq k}\psi_i)$ whose height is at most $\varepsilon_\gamma+2$. Hence, $k-1$ applications of conjunction introduction yields:
$$\vdash^{\varepsilon_\gamma+2+(k-1)}_0\bigwedge_{i\leq k}\mathsf{TI}(\varepsilon_\gamma,\psi_i).$$
Note that:
$$(\varepsilon_\gamma+2+(k-1))\#(\omega+4+n) = \varepsilon_\gamma+\omega+k+n+5.$$
Thus, by Lemma \ref{embedding}:
$$\vdash^{\varepsilon_\gamma+\omega+k+n+5}_r \varphi^\star(\vec{p})$$
where $r:=\mathsf{max}\Big\{\sigma, \mathsf{rk}\big(\bigwedge_{i\leq k}\mathsf{TI}(\varepsilon_\gamma,\psi_i)\big)\Big\}$.

By applying Lemma \ref{elimination}, we infer that 
$$ \vdash^{\varepsilon_{\gamma+1}}_0 \varphi^\star(\vec{p}).$$ Lemma \ref{correctness} implies that $\varphi(\vec{p})$ is true.
\end{proof}


\subsection{The Hard Direction}\label{hard-direction}

With this lemma on board we are ready for the proof of the theorem. Once again, we use Schmerl's technique of reflexive induction; for an overview of this technique see \textsection \ref{reflexive-section}.

\begin{theorem}
$\mathsf{ACA}_0+\mathsf{WO}(\varepsilon_{\alpha}) \vdash \mathsf{RFN}^{1+\alpha}_{\Pi^1_1}(\mathsf{ACA}_0).$
\end{theorem}

\begin{proof}
We let $A(\alpha)$ denote the claim that $\mathsf{ACA}_0+\mathsf{WO}(\varepsilon_{\alpha}) \vdash \mathsf{RFN}^{1+\alpha}_{\Pi^1_1}(\mathsf{ACA}_0) $. We want to prove that $\mathsf{ACA}_0\vdash \forall\alpha A(\alpha)$. By reflexive induction it suffices to prove that $\mathsf{ACA}_0 \vdash \forall\alpha (\mathsf{Pr}_{\mathsf{ACA}_0}(\forall\gamma<\alpha\; A(\gamma)) \rightarrow A(\alpha))$.

\textbf{Reason in} $\mathsf{ACA}_0$. We assume the \emph{reflexive induction hypothesis}, which is that $\mathsf{Pr}_{\mathsf{ACA}_0}(\forall\gamma<\alpha, A(\gamma))$, i.e.,
$$ \mathsf{Pr}_{\mathsf{ACA}_0}\Big(\forall\gamma<\alpha \big(\mathsf{ACA}_0+\mathsf{WO}(\varepsilon_{\gamma})\text{ proves }\mathsf{RFN}^{1+\gamma}_{\Pi^1_1}(\mathsf{ACA}_0)\big)\Big) . $$
We need to show that $A(\alpha)$, i.e., $\mathsf{ACA}_0+\mathsf{WO}(\varepsilon_{\alpha}) \vdash \mathsf{RFN}^{1+\alpha}_{\Pi^1_1}(\mathsf{ACA}_0) $.

To prove the claim we reason as follows:
\begin{flalign*}
\mathsf{ACA}_0 &\vdash \forall\gamma<\alpha (\mathsf{ACA}_0 + \mathsf{WO}(\varepsilon_{\gamma})\text{ proves } \mathsf{RFN}^{1+\gamma}_{\Pi^1_1}(\mathsf{ACA}_0))\\
& \textrm{ by the reflexive induction hypothesis.}\\
\mathsf{ACA}_0 + \mathsf{WO}(\varepsilon_\alpha) &\vdash \forall\gamma<\alpha (\mathsf{ACA}_0 + \mathsf{WO}(\varepsilon_{\gamma})\text{ proves } \mathsf{RFN}^{1+\gamma}_{\Pi^1_1}(\mathsf{ACA}_0))\\
\mathsf{ACA}_0 + \mathsf{WO}(\varepsilon_\alpha) &\vdash \forall\gamma<\alpha \; \mathsf{RFN}_{\Pi^1_1}(\mathsf{ACA}_0+\mathsf{RFN}^{1+\gamma}_{\Pi^1_1}(\mathsf{ACA}_0)) \textrm{ by Lemma \ref{analysis}.}\\
\mathsf{ACA}_0 + \mathsf{WO}(\varepsilon_\alpha) &\vdash \mathsf{RFN}^{1+\alpha}_{\Pi^1_1}(\mathsf{ACA}_0) \textrm{ by definition.}\\
\end{flalign*}
This completes the proof of the theorem.
\end{proof}

As an immediate consequence of $(\bigcirc)$ and ($\bigtriangleup$) we derive the following:

\begin{theorem}
\label{iterations}
$\mathsf{ACA}_0+\mathsf{RFN}^{1+\alpha}_{\Pi^1_1}(\mathsf{ACA}_0) \equiv \mathsf{ACA}_0+\mathsf{WO}(\varepsilon_{\alpha}) $.
\end{theorem}

\subsection{The Main Theorem}\label{main-section}

In the proof of the main theorem we will rely on a theorem typically attributed to Kreisel, namely, that proof-theoretic ordinals are invariant under the addition of true $\Sigma^1_1$ sentences. To state this result, we introduce some notation:

\begin{definition}
$T\vdash^{\Sigma^1_1}\varphi$ if there is a true $\Sigma^1_1$ formula $A$ such that $T+A\vdash \varphi$.
\end{definition}

\begin{definition}
Let $|T|^\star_{\mathsf{WO}}=\mathsf{sup}\{\mathsf{otyp}(\alpha)\mid T\vdash^{\Sigma^1_1} \mathsf{WO}(\alpha)\}$.
\end{definition}

\begin{theorem}[Kreisel]\label{kreisel}
For $\Pi^1_1$-sound $T$ extending $\mathsf{ACA}_0$, $|T|_{\mathsf{WO}}=|T|_{\mathsf{WO}}^\star$.
\end{theorem}

The following lemma appears as Lemma 5.9 in \cite{pakhomov2018reflection}, where it receives a self-contained proof:
\begin{lemma}\label{pwlemma}
If $|T|_{\mathsf{ACA}_0}>\mathsf{otyp}(\alpha)$ then $T\vdash^{\Sigma^1_1} \mathsf{RFN}_{\Pi^1_1}\big(\mathsf{ACA}_0+ \mathsf{RFN}^\alpha_{\Pi^1_1}(\mathsf{ACA}_0)\big)$.
\end{lemma}

We are now ready for the proof of the main theorem.

\begin{theorem}
For any $\Pi^1_1$-sound extension $T$ of $\mathsf{ACA}_0^+$, $|T|_{\mathsf{ACA}_0}=|T|_{\mathsf{WO}}$.
\end{theorem}

\begin{proof}
Let $T$ be a $\Pi^1_1$-sound extension of $\mathsf{ACA}_0^+$.

First, let's see that $|T|_{\mathsf{ACA}_0} \leq |T|_{\mathsf{WO}}$. Let $\mathsf{otyp}(\alpha)<|T|_{\mathsf{ACA}_0}$. We reason as follows:
\begin{flalign*}
T&\vdash^{\Sigma^1_1} \mathsf{RFN}_{\Pi^1_1}\big( \mathsf{ACA}_0+\mathsf{RFN}^\alpha_{\Pi^1_1}(\mathsf{ACA}_0)\big) \text{ by Lemma \ref{pwlemma}.}\\
T&\vdash^{\Sigma^1_1} \mathsf{WO}(\varepsilon_\alpha) \text{ by Theorem \ref{iterations}.}
\end{flalign*}
That is: $$ \mathsf{sup}\{\mathsf{otyp}(\varepsilon_\alpha) | \mathsf{otyp}(\alpha)<|T|_{\mathsf{ACA}_0}\} \leq |T|_{\mathsf{WO}}^\star.$$ 
By Lemma \ref{kreisel}:
$$\mathsf{sup}\{\mathsf{otyp}(\varepsilon_\alpha) | \mathsf{otyp}(\alpha)<|T|_{\mathsf{ACA}_0}\} \leq |T|_{\mathsf{WO}}.$$
It follows that $|T|_{\mathsf{ACA}_0} \leq |T|_{\mathsf{WO}}$.

To see that $|T|_{\mathsf{WO}} \leq |T|_{\mathsf{ACA}_0}$, we reason as follows:
\begin{flalign*}
\mathsf{ACA}_0^+ &\vdash \forall \mathcal{X} (\mathsf{WO}(\mathcal{X}) \rightarrow \mathsf{WO}(\varepsilon_\mathcal{X})) \textrm{ by Theorem \ref{mm}.}\\
\mathsf{ACA}_0^+ &\vdash \mathsf{WO}(\alpha) \rightarrow \mathsf{WO}(\varepsilon_\alpha) \textrm{ by instantiation.}\\
\mathsf{ACA}_0^+ &\vdash \mathsf{WO}(\varepsilon_\alpha) \rightarrow \mathsf{RFN}_{\Pi^1_1}^{1+\alpha}(\mathsf{ACA}_0) \textrm{ by Theorem \ref{iterations}.}\\
\mathsf{ACA}_0^+ &\vdash \mathsf{WO}(\alpha) \rightarrow \mathsf{RFN}_{\Pi^1_1}^{1+\alpha}(\mathsf{ACA}_0) \textrm{ from the previous two lines.}
\end{flalign*}
Let $T$ extend $\mathsf{ACA}_0^+$ and suppose $T \vdash \mathsf{WO}(\alpha)$. Then $T \vdash \mathsf{RFN}_{\Pi^1_1}^{1+\alpha}(\mathsf{ACA}_0)$.

This completes the proof of the theorem.
\end{proof}

\bibliographystyle{plain}
\bibliography{bibliography.bib}

\end{document}